\documentclass[10pt,letterpaper,leqno]{amsart}

\usepackage[T1]{fontenc}
\usepackage{lmodern}
\usepackage{amssymb}
\usepackage{graphicx}
\usepackage{microtype}
\usepackage{hyperref}
\hypersetup{colorlinks=true,linkcolor=blue,citecolor=blue,urlcolor=blue}

\copyrightinfo{}{}

\makeatletter
\AtBeginDocument{%
	\def\@serieslogo{%
		\normalfont\scriptsize
		\href{https://www.ams.org/journals/proc/}{%
			Accepted for publication in the Proceedings of the American Mathematical Society%
		}%
	}%
}
\makeatother
\numberwithin{equation}{section}
\newtheorem{theorem}{Theorem}[section]
\newtheorem{prop}[theorem]{Proposition}
\newtheorem{lemma}[theorem]{Lemma}

\title[Uniqueness of quasi-stationary distributions]{On the uniqueness of quasi-stationary
  distributions for population models with spatial structure}
\author{Pablo Groisman}
\address{Department of Mathematics, University of Buenos Aires,
  1428 Buenos Aires, Argentina; and Instituto de Investigaciones
  Matem{\'a}ticas Luis A. Santal{\'o}, Buenos Aires, Argentina}
\email{pgroisma@dm.uba.ar}
\author{Leonardo T. Rolla}
\address{Department of Statistics, University of S{\~a}o Paulo,
  S{\~a}o Paulo - SP, 05508-220, Brazil}
\email{leonardo.rolla@gmail.com}
\thanks{The second author was partially supported by FAPESP grants
  2023/13453-5 and 2025/27064-6 and CNPq grant 408590/2024-6.}
\author{C{\'e}lio Terra}
\address{Department of Mathematics, Federal University of Minas Gerais,
  Belo Horizonte - MG, 31270-901, Brazil}
\email{caugusto.terra@gmail.com}
\thanks{The third author was partially supported by CAPES grant
  88887.595894/2020-00, FAPEMIG grant 15435 and CNPq grant 151432/2024-4.}

\begin{document}

\begin{abstract}
		 Subcritical population processes are attracted to extinction and do not have non-trivial stationary distributions, which prompts the study of quasi-stationary distributions (QSDs) instead. In contrast to what generally happens for stationary distributions, QSDs may not be unique, even under irreducibility conditions. The general conditions for uniqueness of QSDs are not always easy to check. For the branching process, besides the quasi-limiting distribution there are many other QSDs. In this paper, we investigate whether adding little extra information to the continuous-time branching process is enough to obtain uniqueness. We consider the branching process with genealogy and branching random walks, and show that they have a unique QSD.
	\end{abstract}

\maketitle

	\section{Introduction}
Subcritical population processes in general do not admit non-trivial stationary distributions, but many of them have a \emph{quasi-stationary distribution} (QSD), that is, a distribution that is invariant when conditioned on survival. Formally, let $(\xi_t)_{t \in \mathbb{R}^+}$ be a Markov process on $\Lambda_0:=\Lambda \cup \{\emptyset\}$, with $\emptyset$ an absorbing state that is reached a.s.\ and $\Lambda$ some arbitrary set. Denote by $\mathbb{P}^{\mu}$ the law of the process $(\xi_t)_{t \in \mathbb{R}^+}$ with initial condition sampled from $\mu$ (for $j \in \Lambda_0$ we write $\mathbb{P}^j$ instead of $\mathbb{P}^{\delta_j}$). A measure $\nu$ on $\Lambda$ is a \emph{quasi-stationary distribution} of the process $(\xi_t)_{t \in \mathbb{R}^+}$ if
\[\mathbb{P}^{\nu}(\xi_t \in \cdot \; | \; \xi_t \neq \emptyset)= \nu(\cdot),\]
for all $t>0$. A survey on the main results about QSDs may be found on \cite{MV12}. 

Research on quasi-stationary distributions dates back to~\cite{Yag47}, where Yaglom proved that, for a subcritical discrete-time branching process $(X_n)_{n \in \mathbb{N}}$ with finite variance, the limit
\begin{equation*}
	\lim_{n \to \infty} \mathbb{P}^j(X_n \in \; \cdot \; | \; X_n \neq 0) := \nu( \cdot )
\end{equation*}
does not depend on $j \in \mathbb{N} := \{1,2,\dots\}$, and that $\nu$ is a QSD on $\mathbb{N}$ for the branching process.
Analogous results hold in the continuous-time setting, which is the focus of this paper.

This QSD $\nu$ is the unique QSD with finite mean for the subcritical branching process. It is also the unique \emph{minimal} QSD, that is, a QSD with the minimal mean absorption time among all QSDs. It is known that the subcritical branching process admits a one-parameter family of QSDs, with larger mean absorption times and heavier tails for the population size as the parameter increases~\cite{Cavender78,SVJ66}.
Subsequent characterizations of QSDs for birth-and-death processes show that this family exhausts all QSDs of the subcritical branching process~\cite{Cavender78,Vandoorn91}.
In particular, QSDs need not be unique, in contrast with stationary distributions.

For irreducible Markov processes on finite state spaces, the QSD is unique~\cite{DarrochSeneta1967}.
For countable state spaces, the situation is more delicate. For irreducible processes that do not come down from infinity in finite time, existence of the QSD is equivalent to the existence of an exponential moment of the absorption time~\cite{FKMP95}.
Uniqueness also holds if there exists a finite set to which the process is quickly attracted~\cite{MartinezSanMartinVillemonais2014}.

A general criterion for existence and uniqueness of quasi-stationary distributions was established in \cite{ChaVil16}. It relies on two uniform assumptions formulated at a fixed time horizon and relative to a common reference probability measure. First, the transition kernel of the process, conditioned to the event of non-absorption up to that time, uniformly dominates this reference measure, independently of the initial state. Second, the probability of surviving up to some time, starting from any state, is uniformly comparable to the survival probability when the initial distribution is the reference measure. Together, these assumptions ensure that both the redistribution of mass before absorption and the associated survival weights are controlled in a coherent way. This yields an exponential contraction of the conditioned semigroup in total variation norm, implying existence and uniqueness of the quasi-stationary distribution and exponential convergence toward it. Establishing such a reference measure, however, is generally nontrivial and depends on detailed structural properties of the underlying process.

Variants of birth-and-death processes allowing extinction to occur with more than one alive individual were considered in~\cite{CSVD06}. In this setting, uniqueness of the QSD holds under the same condition for classical birth-and-death processes when the set of states from which extinction can occur is finite, and this assumption can be removed to allow extinction from any state~\cite[Section~5]{ChaVil23}.

Recent results further relate uniqueness of QSDs to exponential integrability properties of absorption and hitting times. If the moment generating function (MGF) of the absorption time is infinite at some point for some initial condition, but there exists a finite set $K$ such that the MGF of the hitting time of $K$ is uniformly bounded over all initial conditions, then the QSD is unique~\cite[Theorem~7.8]{BAJ24}. By contrast, if the MGF of the absorption time is finite on a nontrivial interval but unbounded with respect to the initial condition on that interval, then infinitely many QSDs exist~\cite[Corollary~7.10]{BAJ24}.

All these criteria link uniqueness of QSDs to the speed of absorption.
However, uniqueness of a QSD is not determined by the speed at which the process comes down from infinity.
Indeed, consider the subcritical contact process and subcritical branching process.
They come down from infinity at comparable rates, and they are borderline not coming down from infinity in finite time: the negative drift is roughly proportional to the number of individuals and, although the absorption time for large initial configurations is not bounded, it diverges only logarithmically fast on the population size.
Nevertheless, unlike the subcritical branching process, the subcritical contact process admits only one QSD~\cite{AGR20}.

Whereas the branching process and the contact process both describe the growth of a large population through reproduction and death, in the latter each individual is localized in space with respect to the other individuals.
Besides the exclusion interaction present in the contact process, the addition of a spatial component introduces a kind of memory in the process.
Loosely speaking, each individual is somewhat aware of the presence of other ones nearby.
Thus, we wonder whether adding some geometric structure to the branching process is enough to recover uniqueness, even without interaction between individuals.

We address this question by analyzing two specific models that are defined by enriching the classical branching process with genealogical or spatial information.
Note that they cannot be analyzed by using the general criteria discussed above, since these models have the same speed of absorption as the branching process, and, as discussed above, this is not sufficient to determine uniqueness of the QSD.

\begin{figure}[b!]
	\centering
	\includegraphics[width=.9\textwidth]{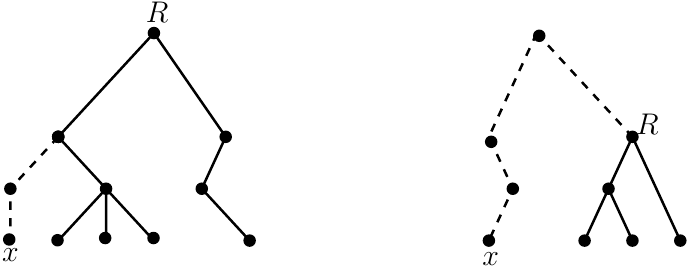}
	\caption{Three examples of pruning.
		In the first two examples, the dashed piece of the tree is deleted when the clock of vertex $x$ rings and $x$ gives birth to zero children. In the third example, the tree has a single vertex $x$, and $x$ is deleted when it gives birth to exactly one child.}
	\label{fig_prunning}
\end{figure}

\bigskip

The \emph{branching process with genealogy} (BPG) introduces genealogical information to the branching process, and is defined as follows (a similar description dates back to~\cite{Harris02}).

Let $\Sigma$ be the set of non-empty finite oriented rooted trees, with the positive direction of the edges pointing away from the root, and $\Sigma_0=\Sigma \cup \{\emptyset\}$ (we consider two trees to be equivalent if there is a graph isomorphism that maps one onto the other preserving the root).
Note that $\Sigma_0$ is a countable space.

A vertex $x$ is a \emph{descendant} of a vertex $y$ or, equivalently, $y$ is an \emph{ancestor} of $x$ if there is an oriented path from $y$ to $x$. By convention, we say that every vertex $x$ is an ancestor and a descendant of itself.

We construct the branching process with genealogy $(\zeta_t)_{t \ge 0}$ as a Markov process on $\Sigma_0$. The idea is that the leaves of $\zeta_t$ correspond to alive individuals at time $t$, and the other vertices are there to keep a record of how the individuals are related. The dynamics of the process is the following. Each leaf $x$ carries an exponential clock that rings with rate $1$ independently of each other. Let $Z$ be a non-negative integer-valued random variable, called the \emph{offspring distribution}. When the clock of vertex $x$ rings, we sample an independent copy $Z^x$ of $Z$. Then, a number $Z^x$ of vertices is added to the tree, each one of them connected to vertex $x$. When this happens, we say that the vertex $x$ \emph{gives birth}, or that a \emph{branching event} occurs.

At certain birth events we will \emph{prune} the tree, so as to keep the property that leaves correspond to alive individuals and to avoid the unlimited growth of the tree with unnecessary information.
Pruning removes all vertices that are no longer needed to determine the affinity between the remaining leaves. More precisely, suppose that the Poisson clock of a leaf $x$ rings at time $t$. We have the following cases. 
\begin{itemize}
	\item
	If $x$ was the only vertex of $\zeta_{t^-}$ and $Z^x=0$, the process is absorbed.
	\item
	If $x$ was the only vertex of $\zeta_{t^-}$ and $Z^x=1$, we remove this vertex and keep only its newborn child (in this case the configuration remains unchanged).
	\item
	If there were multiple vertices in $\zeta_{t^-}$ and $Z^x=0$, we keep only the remaining leaves along with their ancestors, up until their most recent common ancestor $R$, and declare $R$ as the new root.
	\item In all other birth events, there is no pruning.
\end{itemize}
Three examples of pruning are illustrated in Figure~\ref{fig_prunning}.
As a consequence of the pruning, \emph{the root is always the most recent common ancestor of all leaves}.

The projection $(\pi(\zeta_t))_{t \in \mathbb{R}^+}$ of the BPG onto $\mathbb{N}_0$, with $\pi(\zeta_t)$ being the number of leaves of $\zeta_t$, is the usual branching process on $\mathbb{N}_0$ with the same offspring distribution. We note that if we start with a measure on $\Sigma_0$, evolve it by $t$ time units according with the dynamics of the BPG and then project it onto $\mathbb{N}_0$, we obtain the same measure as if we start with a measure in $\Sigma_0$, project it onto $\mathbb{N}_0$ and then evolve it by $t$ time units using the dynamics of a branching process. In other words, the diagram of Figure~\ref{fig_diagrammeasures} is commutative. This implies that, when the offspring distribution's mean is less than one, since the subcritical branching process has no non-trivial invariant probability measures, the same is true for the BPG. Moreover, since $\pi^{-1}\{0\}=\{\emptyset\}$, QSDs for the BPG on $\Sigma$ project onto QSDs on $\mathbb{N}$ with the same mean absorption time. 

\begin{figure}[b!]
	\centering
	\includegraphics[width=.53\textwidth]{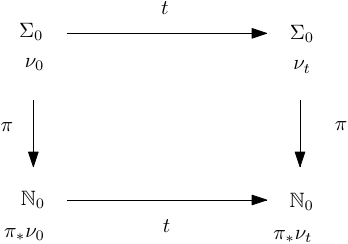}
	\caption{Relationships between measures, $\nu_t$ is the measure of the process at time $t$ and $\pi_*(\cdot)$ is the pushforward measure by $\pi$.}
	\label{fig_diagrammeasures}
\end{figure}

\bigskip

We now define some standard conditions over the offspring distribution that will guarantee existence of a QSD.
We say that the offspring distribution is \emph{subcritical} if $\mathbb{E}[Z]<1$; it is \emph{non-trivial} if $\mathbb{P}(Z=0)>0$ and $\mathbb{P}(Z \geq 2)>0$; and it satisfies the $Z \log Z$ condition if $\mathbb{E}[Z \log^+ Z]<\infty$.

\begin{prop} \label{prop_Yaglom}
	Let $(\xi_t)_{t \ge 0}$ be a continuous-time Markov chain on a countable state-space $\Lambda_0:=\Lambda \cup \{\emptyset\}$ absorbed at $\emptyset$ almost surely, such that $(\xi_t)_{t \ge 0}$ restricted to $\Lambda$ is irreducible.
	Suppose that there is a projection $\pi: \Lambda_0 \rightarrow \mathbb{N}_0$ such that $(\pi(\xi_t))_{t \in \mathbb{R}^+}$ is a branching process with a non-trivial subcritical distribution satisfying the $Z \log Z$ condition.
	Suppose also that $\pi^{-1}\{0\}=\{\emptyset\}$ and that there is $\mathbf{1} \in \Lambda$ such that $\pi^{-1}\{1\}=\{\mathbf{1}\}$. Then there is a QSD $\nu$ of $(\xi_t)_{t \ge 0}$ on $\Lambda$ satisfying, 
	\begin{equation} \label{eq_Yaglom}
		\text{for every $\xi_0, \xi \in \Lambda$, }\lim_{t \rightarrow +\infty}\mathbb{P}^{\xi_0}(\xi_t=\xi \; | \; \xi_t \neq \emptyset) = \nu(\xi).
	\end{equation} 
\end{prop}
A QSD that satisfies the limit~\eqref{eq_Yaglom} is called a \emph{Yaglom limit}.

\bigskip

In contrast to what happens with its projection onto $\mathbb{N}$ (see~\cite{SVJ66}), it turns out that the Yaglom limit is the unique QSD for the BPG. This is the content of our first theorem.

\begin{theorem} \label{thm_unique QSD}
	Let $(\zeta_t)_{t \in \mathbb{R}^+}$ be a branching process with genealogy with non-trivial subcritical offspring distribution satisfying the $Z \log Z$ condition.
	Then $(\zeta_t)_{t \in \mathbb{R}^+}$ has a unique quasi-stationary distribution $\nu$ on $\Sigma$.
\end{theorem}

The main idea behind the proof of Theorem~\ref{thm_unique QSD} is to show that the geometric aspects of the process prevent the existence of more than one QSD. To be more precise, we prove that, conditioned on non-absorption at time $t$, the number of branchings goes to infinity. Once this is done, we investigate the descendants of the first leaf at time $0$ that has descendants alive at time $t$. Either this set of descendants is the entire tree at time $t$ (and thus it is distributed approximately as the Yaglom limit), or the tree at time $t$ has to be very large (for large $t$), contradicting the fact that the distribution is quasi-stationary.

In summary, adding genealogical aspects to the branching process is sufficient for uniqueness of QSD. We now wonder if adding spatial location to the individuals is also sufficient. 

\bigskip

A \emph{branching random walk} (BRW) is a stochastic process that includes characteristics of both branching processes and random walks. Define $\Delta^d:=\{\eta \in \mathbb{N}_0^{\mathbb{Z}^d} : \eta(x) \neq 0 \text{ only for a finite number of } x \in \mathbb{Z}^d\}$. We can see an element $\eta \in \Delta^d$ as a configuration where there are $\eta(x)$ particles at site $x \in \mathbb{Z}^d$. We describe the BRW as a process on $\Delta^d$ in the following way. Let $\lambda>0$ be a parameter. When there is a particle $X$ at site $x$, this particle chooses, at rate $\lambda$, one of its neighbors $y$ uniformly at random and gives birth to a new particle $Y$ at $y$. The particle $Y$ is called a \emph{child} of particle $X$. Each particle dies at rate $1$ and is simply eliminated from the system. We define $\eta_t^A$ as the configuration at time $t$ starting with the configuration $A$ at time $0$. The superscript may be omitted when the initial configuration is clear from context. From now on we assume that $\lambda<1$.

The process as described above does not admit a QSD. The reason is that, conditioned on survival up to time $t$, the configuration $\eta_t^A$ is, in general, very far from the location of the initial configuration $A$. In order to see a QSD, we work with the branching random walk \emph{modulo translations}. Take the equivalence relationship on $\Delta^d$ in which $\eta \sim \eta'$ if $\eta'$ is a translation of $\eta$. Denote by $\langle  \eta \rangle$ the equivalence class of $\eta$. The branching random walk modulo translations is the process $(\langle \eta_t \rangle)_t$ on $\Delta^d_{\sim}:=\Delta^d / \sim$. As the transition probabilities are translation invariant, this process is translation invariant and has the strong Markov property with respect to its natural filtration.

We note that the projection of this process on $\mathbb{N}_0$ is a branching process with offspring distribution supported on $\{0,2\}$ and the critical parameter is $\lambda=1$. Thus, by Proposition~\ref{prop_Yaglom}, there exists at least one QSD $\nu'$ in $\Delta^d_{\sim}$ for the BRW when $\lambda <1$. (The reason why Proposition~\ref{prop_Yaglom} cannot be applied on $\Delta^d$ directly is that $\Delta^d$ has infinitely many configurations with a single particle, whereas $\Delta^d_{\sim}$ bundles all these configurations in the same equivalence class.)

The BRW can be seen as a contact process with multiplicities, i.e., a contact process where it is possible to infect again a site which is already infected. Similarly to what happens to the BPG, its geometric aspects impede the existence of multiple QSDs. This is our second main result.

\begin{theorem} \label{thm_uniqueQSD_BRW}
	Let $(\eta_t)_{t \in \mathbb{R}^+}$ be a branching random walk modulo translations with parameter $\lambda<1$. Then $(\eta_t)_{t \in \mathbb{R}^+}$ has a unique quasi-stationary distribution $\nu'$ on $\Delta^d_{\sim}$.
\end{theorem}

The proof of Theorem \ref{thm_uniqueQSD_BRW} relies on techniques similar to those of Theorem~\ref{thm_unique QSD}. We look for an individual at time $0$ which has alive descendants at time $t$. Either this set of descendants are the only individuals alive at time $t$, and $\langle \eta_t \rangle$ has a distribution close to the Yaglom limit, or there is another individual at time $0$ with alive descendants at time $t$. But the set of descendants of those two individuals are in general very distant from each other at time $t$, and thus the population at time $t$ would be spread over a set with very large diameter. Again, this would be incompatible with the concept of a QSD.

\bigskip

The fact that the branching process with genealogy and the branching random walk both will have to be very ``large'' if there is more than one individual at time $0$ with alive descendants at time $t$ is a consequence of the fact that, conditioned on survival, many birth events will occur in a branching process. This is codified on the following proposition.
\begin{prop} \label{prop_branchings} Let $(X_t)_{t \in \mathbb{R}^+}$ be a subcritical branching process with offspring distribution satisfying the same hypothesis as in Proposition~\ref{prop_Yaglom}. Define, for $k \in \mathbb{N}$ and $t>0$,
	\begin{align*} G_t^X(k)=\left\{
		\begin{alignedat}{2}
			&\text{There are } 0<t_1< \dots <t_{2k} <t;\\
			& X_{t_1}=1, X_{t_2}=2, \dots, X_{t_{2k-1}}=1,X_{t_{2k}}= 2
		\end{alignedat} \right\}.
	\end{align*}
	Then, for every $k \in \mathbb{N}$,
	\begin{align*} \lim_{t\rightarrow +\infty}&\mathbb{P}^{1}(G_t^X(k) \; | \; X_t \neq 0) = 1.
	\end{align*}
\end{prop}

The remaining of this paper is organized as follows. Theorems~\ref{thm_unique QSD} and~\ref{thm_uniqueQSD_BRW} are proved in Sections~\ref{section_QSDbpg} and~\ref{section_QSDBRW} respectively. Propositions~\ref{prop_Yaglom} and~\ref{prop_branchings} are proved using classical results about $\alpha$-positive sub-Markovian kernels in Section~\ref{section_Yaglom}.

\section{The QSD for the branching process with genealogy} \label{section_QSDbpg}

In this section we prove Theorem \ref{thm_unique QSD}.
The \emph{diameter} of a tree $\zeta$ is the length of the longest (non-oriented) path  on the tree, and is denoted by $\operatorname{diam} (\zeta)$.

\begin{proof}[Proof of Theorem~\ref{thm_unique QSD}]
	Fix $t>0$. We define the event
	\[\mathcal{O}_t=\{\text{There is only one leaf at } \zeta_0 \text{ with alive descendants at time } t\}\]
	For every leaf $Y$ of $\zeta_0$, define $\mathcal{D}^Y_t$ as the number of leaves in $\zeta_t$ which are descendants of $Y$. Note that $(\mathcal{D}^Y_t)_{t \ge 0}$ is a branching process with the same offspring distribution of $(\zeta_t)_{t \ge 0}$ and initial condition $1$.
	
	Let $\nu^{*}$ be a QSD of $(\zeta_t)_{t \ge 0}$ on $\Sigma$. For every $\zeta \in \Sigma$,
	\begin{align} 
		\nonumber \nu^{*}(\zeta) = \mathbb{P}^{\nu^{*}}&(\zeta_t=\zeta \; | \; \zeta_t \neq \emptyset)
		\\ \nonumber =\mathbb{P}^{\nu^{*}}&(\zeta_t= \zeta \; | \; \mathcal{O}_t)\mathbb{P}^{\nu^{*}}(\mathcal{O}_t \; | \; \zeta_t \neq \emptyset) 
		\\ \label{eq_decomposution} &+\mathbb{P}^{\nu^{*}}(\zeta_t=\zeta \; | \; \mathcal{O}_t^c \cap \{\zeta_t \neq \emptyset\})\mathbb{P}^{\nu^{*}}(\mathcal{O}_t^c \; | \; \zeta_t \neq \emptyset).
	\end{align}
	Conditioning on $\mathcal{O}_t$, the distribution of $\zeta_t$ is the distribution of a BPG with initial configuration being a tree with a single vertex (which we denote by $\mathbf{1}$)  conditioned on $\{\zeta_t \neq \emptyset \}$. In other words, for every $\zeta \in \Sigma$, $\mathbb{P}^{\nu^{*}}(\zeta_t=\zeta \; | \; \mathcal{O}_t)=\mathbb{P}^{\mathbf{1}}(\zeta_t=\zeta \; | \; \zeta_t \neq \emptyset)$. By Proposition~\ref{prop_Yaglom},
	\begin{equation*}
		\lim_{t \longrightarrow  +\infty} \mathbb{P}^{\mathbf{1}}(\zeta_t=\zeta \; | \; \zeta_t \neq \emptyset) = \nu(\zeta).
	\end{equation*}
	Thus,
	\begin{equation}  \label{eq_limitA_t}
		\lim_{t \longrightarrow  +\infty}\mathbb{P}^{\nu^{*}}(\zeta_t=\zeta \; | \; \mathcal{O}_t)=\nu(\zeta).
	\end{equation}

	On the event $\{\zeta_t \neq \emptyset\}$, there is at least one leaf at time $0$ with alive descendants at time $t$. Choose uniformly at random a leaf $\mathcal{X}$ at time $0$ which has an alive descendant at time $t$.
	Under $\mathbb{P}^{\nu^{*}}$ and conditioned on $\{\zeta_t \neq \emptyset\}$, the distribution of $(\mathcal{D}^\mathcal{X}_t)_t$ is the one of a branching process with initial state $1$ conditioned on survival up to time $t$. Moreover, if  $G_t^{\mathcal{D}^{\mathcal{X}}}(k) \cap \mathcal{O}_t^c \cap \{\zeta_t \neq \emptyset\}$ occurs, then $\zeta_t$ has diameter at least $k$. Indeed, for every $t'<t''$ such that $\mathcal{D}^{\mathcal{X}}_{t'}=1$ and $\mathcal{D}^{\mathcal{X}}_{t''}=2$ there is at least one edge added during the time interval $(t',t'']$ which is still present at time $t$. Hence, when $k$ of those transitions occur, at least $k$ edges have been added. On $\mathcal{O}_t^c\cap \{\zeta_t \neq \emptyset\}$, there are individuals alive at time $t$ which are not descendants of $\mathcal{X}$. By the rules to prune the tree, this implies that $\mathcal{X}$ and those $k$ added edges are still on $\zeta_t$.
	
	On the other hand, by Proposition~\ref{prop_branchings}, for every $\varepsilon>0$, for every $t$ large, $\mathbb{P}^{\mathbf{1}}(G_t^{\mathcal{D}^{\mathcal{X}}}(k)\; | \; \mathcal{D}^{\mathcal{X}}_t \neq 0)> 1-\varepsilon$. Using this fact and the discussion of the previous paragraph, we conclude that, for every $t$ large enough,
	\[
	\mathbb{P}(\{\operatorname{diam}(\zeta_t) < k\} \cap \mathcal{O}_t^c \; | \;\zeta_t \neq \emptyset) < \varepsilon.
	\]
	Thus, for $t$ large enough we have that
	\begin{align*}
		\mathbb{P}^{\nu^{*}}(\mathcal{O}_t^c \; | \; \zeta_t \neq \emptyset) =&\; \mathbb{P}^{\nu^{*}}(\{\operatorname{diam}(\zeta_t) \ge k \} \cap \mathcal{O}_t^c \; | \;  \zeta_t \neq \emptyset) \\&+ \mathbb{P}^{\nu^{*}}(\{\operatorname{diam}(\zeta_t) < k\} \cap \mathcal{O}_t^c\; | \;  \zeta_t \neq \emptyset) \\ 
		\le& \;\mathbb{P}^{\nu^{*}}(\operatorname{diam} (\zeta_t) \ge k \; | \; \zeta_t \neq \emptyset) +\varepsilon \\=&\;\nu^{*}\{\zeta :\operatorname{diam} (\zeta) \ge k\} +\varepsilon.
	\end{align*}
	Letting $k \rightarrow + \infty$ and $\varepsilon \rightarrow 0$, we get
	\begin{equation} \label{eq_limitA_t^c}
		\lim_{t \rightarrow + \infty} \mathbb{P}^{\nu^{*}}(\mathcal{O}_t^c \; | \; \zeta_t \neq \emptyset)=0.
	\end{equation}
	Taking the limit in $t$ in (\ref{eq_decomposution}) and using (\ref{eq_limitA_t}) and (\ref{eq_limitA_t^c}), we have that $\nu(\zeta)=\nu^{*}(\zeta)$.
\end{proof}

\section{Uniqueness of QSD for branching random walks} \label{section_QSDBRW}

In this section, we prove uniqueness of the quasi-stationary distribution for the branching random walk. To do this, we argue that, if there are at time $0$ two individuals with alive descendants at time $t$, the sets of descendants of those two individuals are likely very far from each other, contradicting the fact that the diameter of a QSD distributed configuration cannot grow without bounds.

Theorem~\ref{thm_uniqueQSD_BRW} refers to the process on $\Delta^d_{\sim}$ but the argument is on $\Delta^d$ where objects are more explicit. For that, we will lift the process to $\Delta^d$ and in the end project back to $\Delta^d_{\sim}$, and we will use measures on $\Delta^d_{\sim}$ to sample configurations on $\Delta^d$. To do so, we will choose an arbitrary representative for each element of $\Delta_{\sim}^d$ (for example, putting the first particle in the lexicographical order at the origin). By abuse of notation, we use the same symbol to denote the resulting measure on $\Delta^d$. With this convention, if $\nu$ is a QSD for the process $(\eta_t)_{t \in \mathbb{R}^+}$ on $\Delta^d_{\sim}$, then $\mathbb{P}^{\nu}(\langle \eta_t\rangle = \eta \; | \; \eta_t \neq \emptyset)=\nu(\eta)$, and this is the object that we study in this section. Recall that $\eta_t$ is the process with walks on $\mathbb{Z}^d$ and $\langle \eta \rangle$ is the class of equivalence of the configuration $\eta$.

For a configuration $\eta \in \Delta^d$, define $\operatorname{diam}(\eta)=\operatorname{diam}(\{x \in \mathbb{Z}^d; \eta(x)\ge1\})$. This is constant in each equivalence class, so we can define $\operatorname{diam}(\langle \eta \rangle):=\operatorname{diam}(\eta)$.

A particle $Y$ is said to be a \emph{descendant} of a particle $X$ if there are particles $X=X_0, X_1, \dots, X_k=Y$ such that $X_j$ is a child of $X_{j-1}$ for $1 \le j \le k$.  By convention, we always say that a particle is a descendant of itself. Define, for a particle $X$, $\mathcal{D}_t^{X}$ to be the restriction of the process $\eta$ to the descendants of particle $X$ at time $t$; that is, $\mathcal{D}_t^{X}(x)$ is the number of particles at time $t$ on site $x$ that are descendants of $X$.  
Denote, for each $t>0$, as $\tilde{\mathcal{O}}_t$ the event in which there is exactly one individual at time $0$ with alive descendants at time $t$.

\begin{figure}[b!]
	\centering
	\includegraphics[width=.92\textwidth]{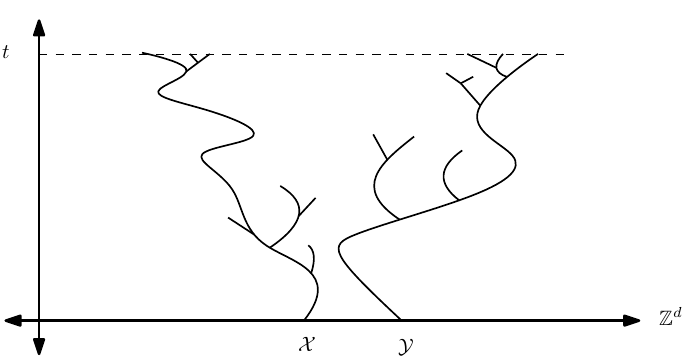}
	\caption{If there are two particles at time $0$ with alive descendants at time $t$, then with high probability there are particles very far from each other at time $t$.}
	\label{fig_descendants}
\end{figure}

We also define $\mathbf{0}$ to be the configurations with no particles.

Fix \(t>0\). On the event \(\{\eta_t \neq \mathbf{0}\}\), consider a particle \(X\) alive at time \(0\) that has alive descendants at time \(t\). We associate to \(X\) a path \((S_X(s))_{s\le t}\), called the \emph{walker of \(X\)}, defined as follows.

Set \(S_X(0)=X\). Suppose \(S_X(s)\) is defined up to some time \(w\le t\), and let \(Y:=S_X(w)\). Let
\[
\tau := \inf\{s\in(w,t] : Y \text{ gives birth at time } s\},
\]
and denote by \(Z\) the particle born from \(Y\) at time \(\tau\).

If \(Z\) has alive descendants at time \(t\), we set \(S_X(s)=Y\) for \(w<s<\tau\) and \(S_X(\tau)=Z\). In this case, if \(Y\) is at site \(x\) and \(Z\) at \(x\pm e_i\) for some \(1\le i\le d\), we say that the walker \emph{jumps} in direction \(\pm e_i\).

If \(Z\) has no alive descendants at time \(t\), we keep \(S_X(s)=Y\) for all \(w<s\le \tau\).

In words, whenever the current walker gives birth, the newborn particle replaces it if and only if it has alive descendants at time \(t\); otherwise, the walker remains unchanged.

The following lemma ensures that the number of jumps of the walker process goes to infinity in probability, conditioned on survival.

\begin{lemma} \label{lemma_jumpwalker}
	For every $M \in \mathbb{N}$ and every $\eta_0 \in \Delta^d$ with only one particle, 
	\[\lim_{t \rightarrow + \infty}\mathbb{P}^{\eta_0}( \text{The walker of } X \text{ jumps at least } M \text{ times before time } t \; | \eta_t \neq \mathbf{0}) =1,\]
	where $X$ is the only particle alive at time $t=0$.
\end{lemma}

We defer the proof of Lemma~\ref{lemma_jumpwalker} to the end of the section.

\begin{proof}[Proof of Theorem \ref{thm_uniqueQSD_BRW}] Conditioning on occurrence of $\tilde{\mathcal{O}}^c_t$ and survival up to time $t$, choose uniformly at random $X$ and $Y$ two distinct individuals with alive descendants at time $t$.
	Conditioning on survival, the descendant processes of distinct initial individuals with live descendants at time $t$ are i.i.d.\ branching processes conditioned on survival up to time $t$.
	We have that $\{\eta_t \neq \mathbf{0}\}\cap \tilde{\mathcal{O}}_t^c \subseteq \{\operatorname{diam} (\langle \eta_t \rangle) \ge |S_{X}(t)-S_{Y}(t)|\}$, for every possible choice of $X$ and $Y$. 
	
	Conditioned on $\{\eta_t \neq \mathbf{0}\} \cap \tilde{\mathcal{O}}_t^c$ and the jump times of $(S_X)_{s \in [0,t]}$, the distribution of the direction of each jump of the walker of $X$ is independent of the direction of the other jumps and uniform in $\{\pm e_1, \pm e_2, \dots, \pm e_d\}$. Therefore, conditioned on survival and on jump times, the jumps of the walker of $X$ are those of a simple symmetric random walk. The same is valid for the jumps of the walker of $Y$, which are also conditionally independent of the jumps of $\mathcal{S}_{X}$, because the process at time $t$ conditioned of survival is the superposition of the independent process of the descendants of the particles alive at $t=0$.
	
	By Proposition~\ref{prop_Yaglom}, there is a Yaglom limit $\nu'$ on $\Delta^d_{\sim}:=\Delta^d / \sim$ for the BRW modulo translations. Let $\tilde{\nu}$ be an arbitrary QSD on $\Delta^d_{\sim}$ for the branching random walk modulo translations.

	Let $\varepsilon>0$. Choose $M \in \mathbb{N}$ satisfying the following conditions: 
	\begin{itemize}
		\item defining $D_{m,M}$ as the event such that the distance between two discrete-time independent simple symmetric random walks at time $m$ starting from arbitrary positions is less than $M^{1/4}$ we have, for every $m \ge M$,
		\begin{equation} \label{ineq_distance}
			\mathbb{P}(D_{m,M}) \le \varepsilon,
		\end{equation}
		with $\mathbb{P}$ being the probability measure in the space where the random walks are constructed;
		\item for the chosen $M$,
		\begin{equation} \label{ineq_diameter}
			\tilde{\nu}\{\eta \in \Delta^d_{\sim}: \operatorname{diam} (\eta) \ge M^{1/4}\} < \varepsilon.
		\end{equation}
	\end{itemize}

	Denote by $\delta$ the configuration on $\Delta^d_{\sim}$ with only one particle. Conditioned on $\tilde{\mathcal{O}}_t$, $\langle \eta_t^{\tilde{\nu}} \rangle$ is distributed as $\langle \eta^{\delta}_t \rangle$.
	
	On the event $\mathcal{\tilde{O}}_t^c \cap \{\eta_t \neq \mathbf{0}\}$, the projection on $\Delta^d_{\sim}$ of each one of the processes $(\mathcal{D}^{X}_s)_{s \in [0,t]}$ and $(\mathcal{D}^{Y}_s)_{s \in [0,t]}$ is distributed as a BRW modulo translations with initial configuration $\delta$ conditioned on survival up to time $t$. By Lemma~\ref{lemma_jumpwalker}, for all $t_0$ big enough, conditioned on survival up to time $t_0$ and on $\tilde{\mathcal{O}}_{t_0}$, with probability at least $1-\varepsilon$, the walkers of $X$ and $Y$ both make at least $M$ jumps before time $t_0$ (call this event $H_{t_0}^M)$. Thus, for $t>{t_0}$,
	\begin{align*}
		\mathbb{P}^{\tilde{\nu}}(\tilde{\mathcal{O}}_{t}^c \; | \; \eta_{t} \neq \mathbf{0}) 
		& \le\mathbb{P}^{\tilde{\nu}}(H_{t_0}^M \; | \; \tilde{\mathcal{O}}_t^c \cap \{\eta_{t} \neq \mathbf{0}\})+ \varepsilon
		\\ & \stackrel{
		}{\le} \mathbb{P}^{\tilde{\nu}}(|S_{X}(t)-S_{Y}(t)| \ge M^{1/4} \; | \; \tilde{\mathcal{O}}_t^c \cap \{\eta_{t} \neq \mathbf{0}\})+2 \varepsilon
		\\ & \le \mathbb{P}^{\tilde{\nu}}(\operatorname{diam} (\langle \eta_t \rangle) \ge M^{1/4} \; | \; \tilde{\mathcal{O}}_t^c \cap \{ \eta_{t} \neq \mathbf{0}\})+2 \varepsilon
		\\ & = \tilde{\nu}\{\eta \in \Delta^d_{\sim}:\operatorname{diam} (\eta) \ge M^{1/4}\}+2 \varepsilon
		\\ & \stackrel{
		}{\le} 3 \varepsilon
	\end{align*}
	(the first, second, and last inequalities are due to Lemma~\ref{lemma_jumpwalker}, \eqref{ineq_distance} and \eqref{ineq_diameter}).
	Since $\varepsilon$ is arbitrary,
	\begin{equation} \label{eq_limA_tBRW}
		\lim_{t \rightarrow +\infty} \mathbb{P}^{\tilde{\nu}}(\tilde{\mathcal{O}}_t^c \; | \; \eta_t \neq \mathbf{0})=0.
	\end{equation}
	Hence, for every $\eta \in \Delta^d_{\sim}$,
	\begin{align}
		\nonumber \tilde{\nu}(\eta)=&\;\mathbb{P}^{\tilde{\nu}}(\langle \eta_t \rangle = \eta \; | \; \tilde{\mathcal{O}}_t)\mathbb{P}^{\tilde{\nu}}(\tilde{\mathcal{O}}_t \; | \; \eta_t \neq \mathbf{0})\\ \nonumber &+\mathbb{P}^{\tilde{\nu}}(\langle \eta_t \rangle = \eta \; | \; \tilde{\mathcal{O}}_t^c \cap \{\eta_t \neq \mathbf{0}\})\mathbb{P}^{\tilde{\nu}}(\tilde{\mathcal{O}}_t^c\; | \; \eta_t \neq \mathbf{0}) \\ 
		\nonumber =&\;\mathbb{P}^{\delta}(\langle \eta_t \rangle =\eta \; | \; \eta_t \neq \mathbf{0})\mathbb{P}^{\tilde{\nu}}(\tilde{\mathcal{O}}_t \; | \; \eta_t \neq \mathbf{0})\\ \nonumber &+\mathbb{P}^{\tilde{\nu}}(\langle \eta_t \rangle =\eta \; | \; \tilde{\mathcal{O}}_t^c \cap \{ \eta_t \neq \mathbf{0}\})\mathbb{P}^{\tilde{\nu}}(\tilde{\mathcal{O}}_t^c \; | \; \eta_t \neq \mathbf{0}).
	\end{align}
	
	Taking the limit in $t$, using (\ref{eq_limA_tBRW}) and the fact that $\nu'$ is a Yaglom limit, we obtain $\tilde{\nu}=\nu'$.
\end{proof}

\begin{proof}[Proof of Lemma \ref{lemma_jumpwalker}.]
	Let $\varepsilon>0$. Choose $M_1$ such that, for every $n \ge M_1$, the probability of at least $n$ successes in $3n$ Bernoulli trials with parameter $1/2$ exceeds $1- \varepsilon/2$, and let $M_2=\max\{3M,3M_1\}$. 
	
	The projection of the process $(\mathcal{D}^X_t)_{t \in \mathbb{R}^+}$ on $\mathbb{N}_0$ given by the total number of particles alive at time $t$ which are descendants of $X$ is a subcritical branching process. Therefore, by Proposition~\ref{prop_branchings}, there is $t_0>0$ such that conditioned on survival at time $t_0$, the process $(\mathcal{D}_t^{X})_{t \in \mathbb{R}^+}$ goes from a state with only one particle to a state with two particles at least $M_2$ times with probability at least $1-\varepsilon/2$.
	
	The above choice of $t_0$ implies that, conditioned on survival up to time $t_0$, the probability that there are at least $M_2$ times when the walker gives birth to a particle exceeds $1-\varepsilon/2$. That is, conditioned on $\{\eta^{\eta_0}_{t_0} \neq \mathbf{0}\}$, the event $\{\text{There are } 0<t_1< \dots< t_{M_2}<t_0; S_{\mathcal{X}}(t_i^-) \text{ gives birth at time } t_i, 1 \le i \le M_2 \}$ is at least $1-\varepsilon/2$. Indeed, every time the configuration has only one particle, this particle must be the walker of $X$; and if at a given time the configuration has more than one particle, this means that there must have been a birth event involving the walker at some earlier time.
	
	Suppose the walker $Z$ of $X$ gives birth at a time $t_1<t_0$ to a particle $W$. By the branching property, the processes $(\mathcal{D}_s^{Z})_{s \ge_{t_1}}$ and $({D}_s^W)_{s \ge _{t_1}}$ (i.e., the processes of the descendants of $Z$ and $W$), conditioned on $\{\eta^{\eta_0}_{t_0} \neq \mathbf{0}\}$ are equally distributed BRWs. As $Z$ is the walker, we know that at least one of those processes is alive at time $t_0$, and by symmetry, conditioning on $\sigma(\mathcal{D}^{X}_s)_{s \le t_1}$ and that the walker at time $t_1$ is $Z$, the probability that $\mathcal{D}_{t_0}^{W} \neq \emptyset$ is at least $1/2$ (the conditioning on survival preserves symmetry between the two descendant subtrees). Thus, each time a walker gives birth, the conditional probability that the walker makes a jump is at least $1/2$. 
	
	By the way $t_0$ was chosen, there is more than $M_2$ times the walker gives birth until time $t_0$ with probability at least $1-\varepsilon/2$ (conditioned on survival). Also conditioning on survival, by the arguments of the previous paragraph and the choice of $M_1$, there is a probability at least $1-\varepsilon/2$ that at least $M$ of those birth events results on a jump of the walker, and the lemma is proved.
\end{proof}

\section{$\alpha$-positiveness, $h$-processes and the Yaglom limit} \label{section_Yaglom}

In this section we prove Proposition \ref{prop_Yaglom}.
Throughout this section, $(\xi_t)_{t \in \mathbb{R}^+}$ is an irreducible Markov chain on $\Lambda_0:=\Lambda \cup \{\emptyset\}$ with $\emptyset$ an absorbing state reached a.s.\ and $\Lambda$ a countable set.

Let $P_t$ be the sub-Markovian kernel associated with the restriction of the process $(\xi_t)_{t \in \mathbb{R}^+}$ to $\Lambda$. That is, for $i,j \in \Lambda$, $P_t(i,j) = \mathbb{P}^i(\xi_t=j)$. By~\cite[Theorem 1]{Kingman63}, for all $i,j \in \Lambda$, the limit
\[\alpha:=\lim_{t \rightarrow + \infty} - \frac{\log P_t(i,j)}{t}\]
exists and does not depend on $i$ or $j$. Furthermore, $\alpha \ge 0$.
We say that the semigroup $(P_t)_t$ (or, equivalently, the process $(\xi_t)_t$) is \emph{$\alpha$-positive} if
\begin{equation*} 
	\limsup_{t \rightarrow +\infty} e^{\alpha t}P_t(i,i) > 0,
\end{equation*}
holds for some (equivalently, for all) $i \in \Lambda$.

We need the following result.
\begin{theorem}[{\cite[Theorem 4]{Kingman63}}] An irreducible sub-Markovian kernel $(P_t)_t$ is $\alpha$-positive if, and only if, there are positive vectors $\nu$ and $h$,  with $\nu h <+\infty$, unique up to a multiplicative constant, such that, for all $t \ge  0$,
	\begin{equation} \label{eq_eigenvectors}
		P_t h = e^{-\alpha t} h; \quad \nu P_t = e^{-\alpha t} \nu.
	\end{equation}
	Moreover, if $(P_t)_t$ is $\alpha$-positive, then for all $i,j$ in $\Lambda$,
	\begin{equation} \label{eq_transitionzeta}
		\lim_{t \rightarrow +\infty} e^{\alpha t}P_t(i,j)=\frac{h_i \nu_j}{\nu h}.
	\end{equation} 
\end{theorem}

If, in addition to $\alpha$-positiveness, the left-eigenvector $\nu$ is summable, existence of the Yaglom limit follows from classical arguments~\cite{SVJ66}, recently rewritten in~\cite[Theorem 3.1]{AEGR15}. Uniqueness of the minimal QSD (i.e., the QSD with the minimal mean absorption time) also follows from~\cite[Theorem 4]{Kingman63}. In this case, we define the \emph{$h$-transform} of $P_t$ as the Markovian kernel $P^h_t$ given by
\[P^h_t(i_0, i_1):=e^{\alpha t}\frac{h_{i_1}}{h_{i_0}}P_t(i_0,i_1).\]

Note that, by \eqref{eq_eigenvectors},
\[\sum_{i_1} P_t^h(i_0,i_1)=\frac{1}{h_{i_0}}e^{\alpha t}\sum_{i_1}h_{i_1}P_t(i_0,i_1)=1,\]
showing that $P^h$ is indeed Markovian.

The $h$-process arises naturally as the limit of the distribution of the process when conditioned on surviving for a long time. More precisely, we have the following lemma.

\begin{lemma} \label{lemma_hprocess}
	If the sub-Markovian kernel $(P_t)_t$ is $\alpha$-positive with summable left eigenvector, then, for all $i_0, i_1 \in \Lambda$,
	\begin{equation*}
		\lim_{s \rightarrow +\infty} \mathbb{P}^{i_0}(\xi_t = i_1 \; | \; \xi_s \neq \emptyset) = P^h_t(i_0,i_1).
	\end{equation*} 
\end{lemma}

\begin{proof} We assume, without loss of generality, that $\sum_j \nu_j=1$ and $\nu h=1$. Using the Markov property of the process $(\xi_t)_{t \in \mathbb{R}^+}$, for $s \ge t$,
	\begin{equation}  \label{eq_limittransitions}
		\mathbb{P}^{i_0}(\xi_t=i_1; \xi_s \neq \emptyset) 
		= \mathbb{P}^{i_0}(\xi_{t}=i_1)\mathbb{P}^{i_1}(\xi_{s-t} \neq \emptyset). 
	\end{equation}
	
	By~\eqref{eq_eigenvectors}, we have that, for every $i,j \in \Lambda$ and $t>0$,
	\begin{equation} \label{eq_dominationmeasure}
		e^{\alpha t} P_t(i,j)=\frac{e^{\alpha t}\nu_iP_t(i,j)}{\nu_i} \le \frac{e^{\alpha t}\sum_i\nu_iP_t(i,j)}{\nu _i}=\frac{\nu_j}{\nu_i}.
	\end{equation}
	Since $\nu$ is summable, using~\eqref{eq_dominationmeasure} and dominated convergence, we can sum over $j$ in~\eqref{eq_transitionzeta} to obtain
	\begin{equation} \label{eq_limabsorptiontime}
		\lim_{s \rightarrow + \infty} e^{\alpha s}\mathbb{P}^{i}(\xi_s \neq \emptyset) = h_i,
	\end{equation} 
	for all $i \in \Lambda$. Using~\eqref{eq_limittransitions} and~\eqref{eq_limabsorptiontime},
	\begin{align}
		\nonumber \lim_{s \rightarrow +\infty}   & \mathbb{P}^{i_0}(\xi_{t} = i_1\; | \; \xi_s \neq \emptyset)
		\\ \nonumber&=\lim_{s \rightarrow +\infty}\frac{\mathbb{P}^{i_0}(\xi_{t}=i_1)\mathbb{P}^{i_1}(\xi_{s-t} \neq \emptyset)e^{\alpha(s-t)}e^{\alpha t}}{\mathbb{P}^{i_0}(\xi_s \neq \emptyset)e^{\alpha s}}
		\\ \label{eq_finitedimensY} &=\mathbb{P}^{i_0}(\xi_{t}=i_1)\frac{h_{i_1}}{h_{i_0}}e^{\alpha t},
	\end{align}
	as we wished to show.
\end{proof}

To prove $\alpha$-positiveness of the process $(\xi_t)_{t \in \mathbb{R}^+}$ we need to adapt the previous concepts to the discrete-time setting. Let $(X_n)_{n \in \mathbb{N}}$ be a discrete-time, irreducible and aperiodic Markov chain on $\Lambda_0$, absorbed a.s.\ at $\emptyset$, with transition matrix $P(\cdot, \cdot)$. Define, for $i \in \Lambda$, $\tau^i:=\inf\{n \in \mathbb{N}: X^i_n=\emptyset \}$, with $X_k^i$ being the chain $(X_n)_n$ starting from state $i$. Analogously to the continuous-time case, we have that, for $i,j \in \Lambda$, $\lim_n (P_n(i,j))^{-1/n}=R\ge1$, and we say that $(X_n)_n$ is \emph{$R$-positive} if $\limsup_n R^{n}P_n(i,j)>0$. 

\begin{proof}[Proof of Proposition~\ref{prop_Yaglom}.] To prove existence of the Yaglom limit, it suffices to prove that the chain $(\xi_t)_t$ is $\alpha$-positive with summable left eigenvector~\cite[Theorem 3.1]{AEGR15}.
	
	The time-discretized chain $(X_k)_{k \in \mathbb{N}}:=(\pi(\xi_k))_{k \in \mathbb{N}}$ is a branching process with offspring distribution $\tilde{Z} \sim X_1^1$.
	By assumption, $\mathbb{E}[Z \log^+ Z]< \infty$.
	Now, by Corollary~III.6.2 of~\cite{AthreyaNey1972}, $\mathbb{E}[\tilde{Z} \log^+ \tilde{Z}]<\infty$.
	Hence, the chain $(X_k)_{k \in \mathbb{N}}$ is $R$-positive~\cite[Section 5]{SVJ66} and thus it has a left eigenvector $\tilde{\nu}$ and a right eigenvector $\tilde{h}$, both unique up to constants, and moreover $\sum_j \tilde{\nu}_j\tilde{h}_j < \infty$.
	Finally, for a branching process, the right eigenvector is given by $\tilde h_j=j$ for every $j$~\cite[p.~308]{Maillard2018}. Therefore $\sum_j \tilde{\nu}_j<\infty$.
	
	Since $\pi^{-1}\{1\}=\{\mathbf{1}\}$, the process $(\xi_k)_{k \in \mathbb{N}}$ is also $R$-positive and has a unique left eigenvector $\nu$.
	Therefore, the continuous-time chain $(\xi_t)_{t \ge 0}$ is $\alpha$-positive with the same eigenvectors. By uniqueness of the left eigenvector, $\tilde{\nu}=\nu \circ \pi^{-1}$. Summability of $\nu$ follows directly from summability of its projection $\tilde{\nu}$ on $\mathbb{N}$. As noted above, this establishes the proposition. 
\end{proof}

\begin{proof}[Proof of Proposition~\ref{prop_branchings}] Let $\varepsilon > 0$. For $N \in \mathbb{N}$ and $T>0$ let $\mathcal{P}_T^N$ be the partition $\mathcal{P}_T^N=\{0,T/N,2T/N, \dots, T\}$ of $[0,T]$. Define $A$ as the event such that there are $t_1, t_2, \dots t_{2k}$  such that $X_{t_1}=1, X_{t_2}=2, \dots, X_{t_{2k-1}}=1, X_{t_{2k}}=2$. Let $A_T$ be the analogous event with all $t_1, \dots, t_{2k} \in [0,T]$ and $A^{N}_T$ be the analogous event with $t_1, \dots, t_{2k} \in \mathcal{P}_T^N$.
	
	By $\alpha$-positiveness of $(X_t)_{t \in \mathbb{R}^+}$, we have that $P^h_t(i,i) \nrightarrow 0$, and thus, under $P^h$, the process $(X_t)_t$ is positive recurrent. Then, $\mathbb{P}^h(A)=1$, with $\mathbb{P}^h$ the law of the process under the kernel $P^h$. As $A_T \nearrow A$, we can find $T$ large enough such that $\mathbb{P}^h(A_T)>1-\varepsilon/2$. For such a fixed $T$, we can find $N$ large enough such that $\mathbb{P}^h(A_T^N) > 1-\varepsilon/2$, because the jump rates at states $1$ and $2$ are constant. By Lemma~\ref{lemma_hprocess}, for $t$ large enough $\mathbb{P}(A^{N}_T \; | \; X_t \neq 0) > \mathbb{P}^h(A^{N}_T)-\varepsilon/2$.
	As $\varepsilon$ is arbitrary, this finishes the proof.
\end{proof}

\section*{Acknowledgments}

Part of this research was carried out while L.R. was visiting the University of Buenos Aires, which is thanked for its hospitality.
The authors thank the anonymous referees for suggestion and criticism which have led to significant improvements in this paper.
\bibliographystyle{amsplain}
\bibliography{bibliography}
\bigskip

\end{document}